\documentclass[10pt, reqno]{amsart} 
\usepackage[T2A]{fontenc}
\usepackage[utf8]{inputenc}
\usepackage[english]{babel}
\usepackage[all]{xy} 
\usepackage{amsmath,amsthm, amssymb,fancyhdr,mathrsfs,color,epsfig, color, graphicx}
\usepackage[usenames,dvipsnames,svgnames,table]{xcolor}
\usepackage[colorlinks=true,hyperindex, pagebackref=true, citecolor=NavyBlue,linkcolor=magenta]{hyperref}
\usepackage{enumitem}
\usepackage{mathrsfs}

\makeatletter
\newcommand\testshape{family=\f@family; series=\f@series; shape=\f@shape.}
\def\myemphInternal#1{\if n\f@shape%
	\begingroup\itshape #1\endgroup\/%
	\else\begingroup\sffamily\small #1\endgroup%
	\fi}
\def\myemph{\futurelet\testchar\MaybeOptArgmyemph}
\def\MaybeOptArgmyemph{\ifx[\testchar \let\next\OptArgmyemph
	\else \let\next\NoOptArgmyemph \fi \next}
\def\OptArgmyemph[#1]#2{\index{#1}\myemphInternal{#2}}
\def\NoOptArgmyemph#1{\myemphInternal{#1}}
\makeatother

\newtheorem{theorem}[subsection]{Theorem}
\newtheorem{lemma}[subsection]{Lemma}

\newtheorem{corollary}[subsection]{Corollary}
\theoremstyle{remark}
\newtheorem{remark}[subsection]{Remark}

\numberwithin{equation}{subsection}

\newcommand\ie{\textit{i.e.\,}}

\newcommand{\Eg}{{E.g.\,}}
\newcommand\RRR{\mathbb{R}}
\newcommand\ZZZ{\mathbb{Z}}
\newcommand{\NNN}{\mathbb{N}}

\newcommand{\KR}{\Gamma}
\newcommand{\Stab}{\mathcal{S}}
\newcommand{\Aut}{\mathrm{Aut}}
\newcommand{\Gstab}{G}
\newcommand{\st}{\mathrm{st}}
\renewcommand{\mod}{\mathrm{mod}}

\newcommand{\Diff}{\mathcal{D}}
\newcommand{\Orb}{\mathcal{O}}
\newcommand{\id}{\mathrm{id}}

\renewcommand{\ker}{\mathrm{Ker}}

\newcommand{\cC}{\mathcal{C}}

\newcommand{\bF}{\mathbf{F}}

\newcommand\nV{{V}}
\newcommand\nW{{W}}

\newcommand{\sV}{\mathsf{\nV}}
\newcommand{\sW}{\mathsf{\nW}}

\newcommand\xxnm{nm} 
\newcommand\xxnmn{nmn} 

\begin{document}
\title[Deformations of smooth functions on $2$-torus]
{Deformations of smooth functions on $2$-torus}

\author{Bohdan Feshchenko}
\address{Topology laboratory, Department of algebra and topology, Institute of Mathematics of National Academy of Science of Ukraine,
Tereshchenkivska, 3, Kyiv, 01601, Ukraine}
\curraddr{}
\email{fb@imath.kiev.ua}

\subjclass[2010]{}

\begin{abstract}
Let $f$ be a Morse function on a smooth compact surface $M$ and $\mathcal{S}'(f)$ be the group of $f$-preserving diffeomorphisms of $M$ which are isotopic to the identity map.
Let also $G(f)$ be a group of automorphisms of the Kronrod-Reeb graph of $f$ induced by elements from $\mathcal{S}'(f)$, and $\Delta'$ be the subgroup of $\mathcal{S}'(f)$ consisting of diffeomorphisms which trivially act on the graph of $f$ and are isotopic to the identity map.
The group $\pi_0\mathcal{S}'(f)$ can be viewed as an analogue of a mapping class group for $f$-preserved diffeomorphisms of $M$.
The groups $\pi_0\Delta'(f)$ and $G(f)$ encode ``combinatorially trivial'' and ``combinatorially nontrivial'' counterparts of $\pi_0\mathcal{S}'(f)$ respectively.
In the paper we compute groups $\pi_0\mathcal{S}'(f)$, $G(f)$, and $\pi_0\Delta'(f)$ for Morse functions on $2$-torus $T^2$.
\end{abstract}

\maketitle

\section{Introduction}\label{sec:Introduction}
Homotopy properties of Morse functions on surfaces were studied by many authors.
\Eg connected components of the space of Morse functions were computed in the unpublished paper by H.~Zieschang, by S.~Matveev in the paper by E.~Kudryavtseva~\cite{Kudryavtseva:MatSb:1999}, and V.~Sharko~\cite{Sharko:PrIntMat:1998}, cobordism groups of the space of Morse functions on surfaces were described by K.~Ikegami and O.~Saeki~\cite{IkegamiSaeki:JMSJap:2003}, and B.~Kalmar~\cite{Kalmar:KJM:2005}.
Homotopy groups of stabilizers and orbits of Morse functions on surfaces with respect to the action of diffeomorphism groups were studied by S.~Maksymenko~\cite{Maksymenko:AGAG:2006, Maksymenko:MFAT:2010, Maksymenko:ProcIM:ENG:2010, Maksymenko:UMZ:ENG:2012, Maksymenko:DefFuncI:2014} and E.~Kudryavtseva \cite{Kudryavtseva:MathNotes:2012, Kudryavtseva:MatSb:2013}.
We will give an overview of these results.

Let $M$ be a smooth compact surface and $X$ be a closed (possible empty) subset of $M$.
The group $\Diff(M,X)$ of diffeomorphisms fixed on some neighborhood of $X$ acts on the space of smooth functions $C^{\infty}(M)$ by the rule: $C^{\infty}(M)\times\Diff(M,X)\to C^{\infty}(M)$, $(f,h) \mapsto f\circ h$.
With respect to this action we denote by
\begin{gather*}
	\mathcal{S}(f,X) = \{h\in \Diff(M,X)\mid  f\circ h = f\},\\
	\Orb(f,X) = \{f\circ h\mid  h\in\Diff(M,X)\}
\end{gather*}
the stabilizer and the orbit of $f\in C^{\infty}(M)$ respectively.
Endow strong Whitney $C^{\infty}$-topologies on $C^{\infty}(M)$ and $\Diff(M,X)$.
Then for a $f\in C^{\infty}(M)$ these topologies induce some topologies on $\Stab(f,X)$ and $\Orb(f,X)$.
We denote by $\Diff_{\id}(M,X)$, $\Stab_{\id}(f,X)$, and $\Orb_f(f,X)$ connected components of the identity map $\id_M$ of $\Diff(M,X)$, $\Stab(f,X)$, and the component of $\Orb(f,X)$ containing $f$ respectively.
If $X=\varnothing$ we will omit the symbol ``$\varnothing$ '' from our notations, \ie, set $\Diff(M):=\Diff(M,\varnothing)$, $\Stab(f) := \Stab(f,\varnothing)$, $\Orb(f):= \Orb(f,\varnothing)$, and so on.

By a \myemph{Morse function} $f$ on a $M$ we will mean a smooth function which satisfies the following conditions:
\begin{itemize}[leftmargin=*]
	\item all critical points of $f$ are non-degenerate and belong to the interior of $M$;
	\item the function $f$ takes constant values on each boundary component of $M$.
\end{itemize}

Notice that if $N\subset M$ is a subsurface whose boundary components are regular components of some level-sets of a Morse function $f:M\to\RRR$, then the restriction $f|_{N}$ is a Morse function as well in the sense of the above definition.

\begin{theorem}\label{thm:homotopy-orbits}
	{\rm\cite{Sergeraert:ASENS:1972,Maksymenko:AGAG:2006,Maksymenko:UMZ:ENG:2012,Maksymenko:OsakaJM:2011}.}
	Let $f$ be a Morse function on a smooth compact surface $M$, and $X$ be a closed (possibly empty) subset of $M$ consisting of finitely many connected components of some level-sets of $f$ and some critical points of $f$.
	Then the following statements hold.
	\begin{enumerate}[wide, label={\rm(\arabic*)}, itemsep=1ex]
		\item
		The map $p:\mathcal{D}_{\id}(M,X)\to \mathcal{O}(f,X)$ defined by $p(h) = f\circ h$ is a Serre fibration with the fiber $\mathcal{S}(f,X)$.
		Hence $p(\mathcal{D}_{\id}(M)) = \mathcal{O}_f(f)$, and the restriction $p|_{\mathcal{D}_{\id}(M)}: \mathcal{D}_{\id}(M)\to \mathcal{O}_f(f,X)$ is also a Serre fibration with the fiber $\Stab'(f,X) = \Stab(f)\cap \Diff_{\id}(M,X)$.
		
		\item
		$\mathcal{O}_f(f,X) = \mathcal{O}_f(f, X\cup \partial M)$, so $\pi_k(\mathcal{O}_f(f,X)) = \pi_k(\mathcal{O}_f(f, X\cup \partial M))$ for $k\geq 1$.
		
		\item
		Suppose that either $f$ has a saddle point or $M$ is a non-orientable surface.
		Then $\mathcal{S}_{\id}(f)$ is contractible, $\pi_k\mathcal{O}_f(f) = \pi_k M$, $k\geq 3$, $\pi_2\mathcal{O}_f(f) = 0$, and for $\pi_1\mathcal{O}_f(f)$ we have the following short exact sequence of groups:
		\begin{equation}
			\label{eq:pi1-main}
			\xymatrix{
				\pi_1\Diff_{\id}(M) \
				\ar@{^{(}->}[r]^{p_1} &
				\ \pi_1\Orb_f(f) \
				\ar@{->>}[r]^{\partial_1} &
				\ \pi_0 \mathcal{S}'(f)
			}\footnote[1]{Throughout the text injective and surjective maps of groups will be also denoted by hooked $\hookrightarrow$ and double-headed arrows $\twoheadrightarrow$ respectively.}.
		\end{equation}
		Moreover the group $p_1(\pi_1\mathcal{D}_{\id}(M))$ is contained in the center of $\pi_1\mathcal{O}_f(f)$.
		
		\item
		If $\chi(M) < |X|$, then $\mathcal{D}_{\id}(M,X)$ is contractible, $\pi_k\mathcal{O}_f(f,X) = 0$ for $k\geq 2$, and the boundary map
		\[ \partial_1: \pi_1\mathcal{O}_f(f,X)\longrightarrow \pi_0\mathcal{S}'(f,X) \]
		is an isomorphism.
	\end{enumerate}
\end{theorem}
We recall the definition of the map $\partial_1$.
Let $\omega:[0,1]\to \mathcal{O}_f(f)$, $\omega_0 = \omega_1$ be a loop in $\mathcal{O}_f(f)$ based in $f$.
Since $p$ is a Serre fibration, it follows that there exists an isotopy $h:M\times [0,1]\to M$ such that $\omega_t = f\circ h_t$, $h_0 = \id$, and $h_1\in \mathcal{S}'(f)$, \ie, $h_1$ is such that $f\circ h_1 = f$.
Then the map $\partial_1$ is defined by the formula $\partial([\omega]) = [h_1]\in \pi_0\mathcal{S}'(f)$.

\subsection{Automorphisms of graphs of functions on surfaces}\label{sec:Aut-graphs}
Let $f:M\to \RRR $ be a Morse function on a smooth compact oriented surface $M$ and $c$ be a real number.
A connected component $C$ of the level-set $f^{-1}(c)$ is called {\it critical}, if $C$ contains at most one critical point of $f$, otherwise $C$ is {\it regular}.
Let $\Xi$ be a partition of $M$ into connected components of level-sets of $f$.
It is well known that the quotient-space $\KR_f = M/\Xi$ has a structure of a $1$-dimensional CW complex called {\it the graph of $f$, or Kronrod-Reeb graph} of $f$.
Let also $p_f:M\to \Gamma_f$ be a projection map.
Then $f$ can be represented as the composition:
\[
\xymatrix{
	f = \widehat{f} \circ p_f:& M \ar[r]^-{p_f} & \KR_f \ar[r]^-{\widehat{f}} & \RRR.
}
\]
Denote by $\Aut(\KR_f)$ the group of homeomorphisms of the graph $\KR_f$.
Note that each $h\in \Stab(f,X)$ preserves level-sets of $f$.
Hence, $h$ induces the homeomorphism $\rho(h)$ of $\KR_f$ such that the following diagram
\[
\xymatrix{
	M \ar[rr]^-{p_f} \ar[d]_h && \KR_f \ar[rr]^-{\widehat{f}} \ar[d]_{\rho(h)} && \RRR \ar@{=}[d] \\
	M \ar[rr]^-{p_f}          && \KR_f \ar[rr]^-{\widehat{f}}                  && \RRR
}
\]
commutes, and the correspondence $h\mapsto \rho(h)$ is a homomorphism \[\rho:\Stab(f,X)\to \Aut(\KR_f).\]
One can check that the image $\rho(\Stab(f,X))$ is a finite subgroup in $\Aut(\KR_f)$.
The image $\rho(\Stab'(f,X))$ in $\Aut(\KR_f)$ will be denoted by $\Gstab(f,X)$.

Let $\Delta(f,X)$ be the normal subgroup of $\mathcal{S}(f,X)$ consisting of diffeomorphisms which leave invariant every connected component of each level set of $f$, and $\Delta'(f,X)$ be the following intersection $\Delta(f,X)\cap \mathcal{D}_{\id}(M,X)$.
It is known that $\pi_0\Delta'(f,X)$ is a free abelian group and $\ker\rho = \pi_0\Delta'(f,X)$. So the following sequence of groups is exact
\begin{equation}
	\label{eq:S'(f)}
	\xymatrix{
		\pi_0\Delta'(f,X) \ar@{^{(}->}[r]^-{j_0} & 	\pi_0\mathcal{S}'(f,X) \ar@{->>}[r]^-{\rho} & G(f,X),
	}
\end{equation}
see~\cite[Section~4]{Maksymenko:DefFuncI:2014}.
From \cite[Theorem 5.2]{Maksymenko:ProcIM:ENG:2010}, we have another short exact sequence
\begin{equation}
	\label{eq:G-sec}
	\xymatrix{
		\pi_1\mathcal{D}_{\id}(M,X)\times \pi_0\Delta'(f,X) \
		\ar@{^{(}->}[r]^-{\iota_1} &
		\ \pi_1\mathcal{O}_f(f,X) \
		\ar@{->>}[r]^-{\rho\circ \partial_1} &
		\ G(f,X),
	}
\end{equation}
in which $\iota_1$ is defined as follows.
Let $\alpha\in\pi_1\mathcal{D}_{\id}(M,X)$ be an element represented by some loop $\{h^t:M \to M\}_{t\in[0,1]}$ in $\mathcal{D}_{\id}(M,X)$ such that $h^0=h^1=\id_{M}$, and $\phi\in \Delta'(f,X)=\Delta(f) \cap \Diff_{\id}(M,X)$.
Fix any isotopy $\{\phi^t:M\to M\}_{t\in[0,1]}$ between $\id_{M} = \phi^0$ and $\phi = \phi^1$.
Then
\[
\iota_1(\alpha,\phi) = [ f \circ h^t \circ \phi^t] \in \pi_1\mathcal{O}_f(f,X).
\]


\subsection{Main diagram}
Thus for a given Morse function $f$ on a smooth compact oriented surface $M$ we considered several spaces associated with $f$.
If  $M\neq S^2$, then all non-trivial homotopy information is encoded in the following commutative diagram:
\begin{equation}\label{diag:main-diag}
	\begin{gathered}
		\xymatrix{
			\pi_1\mathcal{D}_{\id}(M) \times \pi_0\Delta'(f) \ar@{->>}[rr]^{\mathrm{pr}_2} \ar@{->>}[d]_{\mathrm{pr}_1} \ar@{^{(}->}[rd]^{\iota_1}&&\pi_0\Delta'(f) \ar@{^{(}->}[d]_{j_0}\\
			\pi_1\mathcal{D}_{\id}(M) \ar@{^{(}->}[r]^{p_1} & \pi_1\mathcal{O}_f(f) \ar@{->>}[r]^{\partial_1} \ar@{->>}[rd]_{\rho\circ\partial_1}& \pi_0\mathcal{S}'(f) \ar@{->>}[d]_{\rho}\\
			&& G(f),
		}
	\end{gathered}
\end{equation}
where $\mathrm{pr}_1$ and $\mathrm{pr}_2$ are projections on the first and the second factor.
In diagram~\eqref{diag:main-diag} horizontal, vertical and diagonal sequences coincide with sequences~\eqref{eq:pi1-main}, \eqref{eq:S'(f)}, and~\eqref{eq:G-sec} respectively.

Let
\begin{gather}
	\label{eq:fst-seq}
	\xymatrix@1{A_1 \ \ar@{^{(}->}[r]^-{i_1} & \ A_2 \ \ar@{->>}[r]^-{p_1} & \ A_3}, \\
	\label{eq:scd-seq}
	\xymatrix@1{B_1 \ \ar@{^{(}->}[r]^-{i_2} & \ B_2 \ \ar@{->>}[r]^-{p_2} & \ B_3}
\end{gather}
be two exact sequences of groups.
Recall that sequences~\eqref{eq:fst-seq} and~\eqref{eq:scd-seq} are \myemph{isomorphic}  if there exist isomorphisms $\phi = \{\phi_i:A_i\to B_i$, $i = 1,2,3\}$ such that the following diagram commutes
\[
\xymatrix{
	A_1\ar@{^{(}->}[r]^{i_1} \ar[d]^{\phi_1}& A_2 \ar@{->>}[r]^{p_1} \ar[d]^{\phi_2}& A_3 \ar[d]^{\phi_3}\\
	B_1\ar@{^{(}->}[r]^{i_2} & B_2 \ar@{->>}[r]^{p_2} &B _3.
}
\]
Similarly one can define the notion of an isomorphism for commutative diagrams.

The main aim of the paper is to describe diagram~\eqref{diag:main-diag} for Morse functions on $2$-torus up to an isomorphism.

\subsection{Acknowledgments} The author is grateful to Sergiy Maksymenko for useful discussions.

\subsection{Structure of the paper}
Section~\ref{sec:wreath-product} collects definitions of wreath products, which we need to state our main result, Theorem~\ref{thm:main}.
We recall some known results about Morse functions on $2$-torus and their graphs in Section~\ref{sec:function-graphs}, and the fundamental groups of such functions is described in Section~\ref{sec:Orbits}.
Section~\ref{sec:prelim} contains some facts needed for the proofs of our results, and we will prove Theorem \ref{thm:main} in Sections~\ref{sec:proof-1} and~\ref{sec:proof-2}.

\section{Wreath products}\label{sec:wreath-product}
To state our results we need special kinds of wreath products of groups with cyclic groups which we describe below.
Let $G$ be a group and $n,m\geq 1$ be integers.
We will consider the following wreath products:
\begin{itemize}[itemsep=1ex]
	\item $G\wr_n\ZZZ:= G^n\rtimes_{\alpha} \ZZZ$,
	\item $G\wr \ZZZ_n:= G^n\rtimes_{\beta}\ZZZ_n$,
	\item $G\wr_{n,m}\ZZZ^2:= G^{{nm}}\rtimes_{\gamma}\ZZZ^2$,
	\item $G\wr (\ZZZ_n\times \ZZZ_m):= G^{{nm}}\rtimes_{\delta} (\ZZZ_n\times \ZZZ_m)$,
\end{itemize}
where $\alpha:G^n\times\ZZZ\to G^n$ and $\beta:G\times \ZZZ_n\to G^n$ correspond to a non-effective $\ZZZ$-action and an effective $\ZZZ_n$-action on $G^n$ by cyclic shifts of coordinates defined by formulas:
\begin{align*}
	\bigl( (g_i)_{i = 0}^{n-1},a\bigr)  & \stackrel{\alpha}{\longmapsto} (g_{i+a})_{i = 0}^{n-1}, &
	\bigl( (g_i)_{i = 0}^{n-1},b \bigr) & \stackrel{\beta}{\longmapsto}  (g_{i+b})_{i = 0}^{n-1},
\end{align*}
where all indexes are taken modulo $n$, $g_i\in G$, $a\in \ZZZ$, $b\in \ZZZ_n$.
Similarly $\gamma:G^{{\xxnm}}\times \ZZZ^2\to G^{{\xxnm}}$ and $\delta:G^{{\xxnm}}\times (\ZZZ_n\times \ZZZ_m)\to G^{{\xxnm}}$ correspond to a non-effective $\ZZZ^2$-action and an effective $\ZZZ_n\times\ZZZ_m$-action on $G^{{\xxnm}}$ by cyclic shifts of the corresponding coordinates defined by formulas
\begin{align*}
	\bigl( (g_{ij})_{i,j = 0}^{n-1, m-1}, (a,b)   \bigr) & \stackrel{\gamma}{\longmapsto} (g_{i+a,j+b})_{i,j = 0}^{n-1, m-1}, \\
	\bigl( (g_{ij})_{i,j = 0}^{n-1, m-1}, (a',b') \bigr) & \stackrel{\delta}{\longmapsto} (g_{i+a',j+b'})_{i,j = 0}^{n-1, m-1},
\end{align*}
where the indexes $i$ and $j$ takes modulo $n$ and $m$ respectively, $(a,b)\in \ZZZ^2$, $(a', b')\in \ZZZ_n\times\ZZZ_m$.

So $G\wr_n\ZZZ$ and $G\wr\ZZZ_n$ are direct products of sets $G^n\times \ZZZ$ and $G^n\times \ZZZ_n$ with the following multiplications
\begin{align*}
	(g,a)\cdot (g',a') &= (\alpha(g, a')g', a+a'), &
	(g,a)\cdot (g',a') &= (\beta(g, b')g', b+b'),
\end{align*}
for $g,g'\in G^n$, $a,a'\in\ZZZ$, and $b,b'\in\ZZZ_n$.
Similarly $G\wr_{n,m}\ZZZ^2$ and $G\wr (\ZZZ_n\times\ZZZ_m)$ are direct products of sets $G^{nm}\times\ZZZ^2$ and $G^{nm}\times\ZZZ_n\times\ZZZ_m$ respectively with multiplications
\begin{align*}
	(g, (a,b))\cdot(g', (a',b')) &= \bigl(\gamma(g, a',b')g', (a + a', b+ b')\bigr),\\
	(g, (c,d))\cdot(g', (c',d')) &= \bigl(\gamma(g, c',d')g', (c + c', d+ d')\bigr)
\end{align*}
for $g,g'\in G^{{\xxnm}}$, $a,a'b,b'\in\ZZZ$, $c,c'\in\ZZZ_n$, and $d,d'\in\ZZZ_m$.

The general definition of wreath product and its properties the reader can find in~\cite{Meldrum:Longman:1995}.

\section{Main result}\label{sec:main-result}
Let $f$ be a Morse function on $T^2$, $\KR_f$ be its graph, and $p_f:T^2\to \KR_f$ be the projection map induced by $f$.
\begin{lemma}
	The map $p_f^*:\pi_1 T^2\to \pi_1 \Gamma_f$ induced by $p_f$ is an epimorphism with a nonzero kernel.
	Hence the graph of $f$ is either a tree, or has a unique circuit.
\end{lemma}
\begin{proof}
	Let $x\in T^2$ be an arbitrary point.
	Note that $\pi_1(\Gamma_f, p_f(x)) = F_k$, where $F_k$ is a free groups of some rank $k$ equals the number of independent cycles in $\Gamma_f$, and $\pi_1(T^2, x)$ is isomorphic to $\ZZZ^2$.
	It is easy to see that there exists a map $s:(\Gamma_f, p_f(x))\to (T^2, x)$ such that the composition $p_f\circ s$ is homotopic to the identity map $\id_{\Gamma_f}$ relatively base point.
	Then the composition
	\[
	\xymatrix{
		\pi_1 (\Gamma_f, p_f(x)) \ar[r]^-{s*} & \pi_1 (T^2, x) \ar[r]^-{p_f^*} & \pi_1 (\Gamma_f, p_f(x))
	}
	\]
	is the identity isomorphism.
	
	Moreover $p_f^*$ is a surjective map, and $\pi_1 (\Gamma_f, p_f(x)) = F_k$ is a subgroup of the group $\pi_1 (T^2, x)=\ZZZ^2$.
	Since $\pi_1 (T^2,x)$ is commutative, it follows that $F_k$ as a subgroup of $\pi_1 (T^2, x)$ is also a commutative group.
	This is possible only when $k = 0$ or $k = 1$.
	So $\Gamma_f$ is either a tree if $k = 0$, or $\Gamma_f$ has a unique circuit if $k = 1$.
	In both of these cases $p_f^*$ has a nonzero kernel.
\end{proof}
\noindent Denote by $\mathscr{F}_0$ the class of Morse functions on $T^2$ whose graphs are trees, and by $\mathscr{F}_1$ the class of Morse functions on $T^2$ whose graphs contain circuits\footnote[1]{Thus the index $i$ in $\mathscr{F}_i$ refers to the rank of homology group $H_1(\Gamma_f,\ZZZ)$.}.

Our main result is the following theorem.

\begin{theorem}\label{thm:main}
	Let $f$ be a Morse function on $T^2$ and $\Gamma_f$ be its graph.
	\begin{enumerate}[wide, label={\rm(\arabic*)}]
		\item\label{enum:thm:main:1}
		Assume that $f$ belongs to $\mathscr{F}_0$.
		Then there exists a set of mutually disjoint $2$-disks $\mathbb{D} = \{D_i \}_{i= 1}^r\subset T^2$ for some $r\geq 1$ such that each restriction $f|_{D_i}:D_i\to\RRR$, $i = 1,\ldots, r$ is also a Morse function, and the diagram~\eqref{diag:main-diag} is isomorphic to
		
		\begin{equation}\label{diag:main-diag-tree}
			\begin{gathered}
				\xymatrix{
					\pi_1\mathcal{D}_{\id}(T^2) \times (\Delta_{\mathbb{D}})^{{\xxnm}n} \ar@{->>}[rr]^{pr_2} \ar@{->>}[d]_{pr_1} \ar@{^{(}->}[rd]^{\iota_1}&& (\Delta_{\mathbb{D}})^{{\xxnm}n} \ar@{^{(}->}[d]^{j_0}\\
					\pi_1\mathcal{D}_{\id}(T^2) \ar@{^{(}->}[r]^{p_1} & \mathcal{S}_{\mathbb{D}}\wr_{n, \xxnm} \ZZZ^2 \ar@{->>}[r]^{\partial_1} \ar@{->>}[rd]_{\partial_1\circ \rho}& \mathcal{S}_{\mathbb{D}}\wr(\ZZZ_{n}\times \ZZZ_{\xxnm}) \ar@{->>}[d]_{\rho}\\
					&& G_{\mathbb{D}}\wr (\ZZZ_{n}\times \ZZZ_{\xxnm})
				}
			\end{gathered}
		\end{equation}
		for some $n,m\in \NNN$, where
		\begin{align*}
			\Delta_{\mathbb{D}}      \! &=\!  \prod_{i = 1}^r \pi_0\Delta'(f|_{D_i},\partial D_i), &
			\mathcal{S}_{\mathbb{D}} \! &=\!  \prod_{i = 1}^r \pi_0\mathcal{S}'(f|_{D_i}, \partial D_i), &
			G_{\mathbb{D}}           \! &=\!  \prod_{i = 1}^r G(f|_{D_i}).
		\end{align*}
		
		\item\label{enum:thm:main:2}
		Assume that $f$ belongs to $\mathscr{F}_1$.
		Then there exists a set of mutually disjoint $2$-disks $\mathbb{Y} = \{Y_{ij}\}_{i = 0,\ldots, k}^{j = 0,\ldots, c_i}\subset T^2$ for certain $k,c_i\in \NNN$, and $m_i\in \NNN$, $i = 0,\ldots, k-1$, such that diagram~\eqref{diag:main-diag} is isomorphic to
		\[
		\xymatrix{
			\pi_1\mathcal{D}_{\id}(T^2)\times (\Delta_{\mathbb{Y}}^n )/H \ar@{->>}[rr]^{pr_2} \ar@{->>}[d]_{pr_1} \ar@{^{(}->}[rd]^{\iota_1}&& (\Delta_{\mathbb{Y}}^n)/H\ar@{^{(}->}[d]^{j_0}\\
			\pi_1\mathcal{D}_{\id}(T^2) \ar@{^{(}->}[r]^{p_1} & \mathcal{S}_{\mathbb{Y}}\wr_{n} \ZZZ \ar@{->>}[r]^{\partial_1} \ar@{->>}[rd]_{\partial_1\circ \rho} & (\mathcal{S}_{\mathbb{Y}}\wr \ZZZ_n)/H \ar@{->>}[d]_{\rho}\\
			&& G_{\mathbb{Y}}\wr \ZZZ_n,
		}
		\]
		for some $n\in \NNN$,
		where
		\begin{align*}
			\Delta_{\mathbb{Y}} &= \prod_{i = 1}^k ( (\prod_{j = 1}^{c_i} \Delta_{Y_{ij}})^{m_i} \times m_i\ZZZ), &
			\Delta_{Y_{ij}}     &= \pi_0\Delta'(f|_{Y_{ij}}, \partial Y_{ij}), \\
			\mathcal{S}_{\mathbb{Y}} &= \prod_{i = 1}^k ( (\prod_{j = 1}^{c_i} \mathcal{S}_{Y_{ij}} ) \wr_{m_i}\ZZZ), &
			\mathcal{S}_{Y_{ij}}     &= \pi_0\mathcal{S}'(f|_{Y_{ij}}, \partial Y_{ij}), \\
			G_{\mathbb{Y}} &= \prod_{i = 1}^k ( (\prod_{j = 1}^{c_i} G_{Y_{ij}} ) \wr\ZZZ_{m_i}), &
			G_{Y_{ij}}     &= G(f|_{Y_{ij}}, \partial Y_{ij}),
		\end{align*}
		and the group $H$ is a normal subgroup of $\Delta_{\mathbb{Y}}^n$ isomorphic to $\ZZZ$ and generated by the element
		\begin{equation}\label{eq:Gars}
			(\underbrace{(E_1, m_1,E_2,m_2,\ldots, E_k, m_k), \ldots, (E_1, m_1,E_2,m_2,\ldots, E_k, m_k)}_{n}),
		\end{equation}
		where
		$E_i$ is the unit of the group $(\prod_{j = 1}^{c_i} \Delta_{Y_{ij}})^{m_i}$, $i = 0,1\ldots, k$.
		The group $H$ is contained in the center of $\mathcal{S}_{\mathbb{Y}}\wr_n\ZZZ$ via  ${j_0}(h) = (h,0)$ for $h\in H$.
	\end{enumerate}
\end{theorem}

\begin{remark}\label{rem:general}
	\begin{enumerate}[wide]    
		\item\label{rem:general:1}
		It is known that the inclusion $T^2\hookrightarrow \mathcal{D}_{\id}(T^2)$ is a homotopy equivalence, see \cite{EarleEells:DG:1970,Gramain:ASENS:1973}. So, the group $\pi_1\mathcal{D}_{\id}(T^2)$ is isomorphic to $\ZZZ^2$.
		
		\item\label{rem:general:2}
		S. Maksymenko showed that the group $\pi_1\mathcal{O}_f(f)$ is a subgroup of the braid group of $T^2$, see for example \cite[Theorem 1.1]{Maksymenko:DefFuncI:2014}.
		The element \eqref{eq:Gars} can be regarded as a certain analogue of the Garside element in this braid group.
		
		\item\label{rem:general:3}
		In a series of papers~\cite{MaksymenkoFeshchenko:2014:HomPropCycle, MaksymenkoFeshchenko:2014:HomPropTree, MaksymenkoFeshchenko:2015:HomPropCycleNonTri, Feshchenko:2014:HomPropTreeNonTri} S.~Maksymenko and the author described groups $\pi_1 \mathcal{O}_f(f)$ for functions on $2$-torus.
		We shortly review these results in Section~\ref{sec:Orbits}.
		
		\item\label{rem:general:4}
		Since the structure of groups $\pi_1\mathcal{D}_{\id}(T^2)$ and $\pi_1\mathcal{O}_f(f)$ are known, it follows that for description of diagram \eqref{diag:main-diag} we need to find a structure of groups $\pi_0\mathcal{S}'(f)$, $G(f)$, and $\pi_0\Delta'(f)$.
	\end{enumerate}
\end{remark}

\section{Functions on $2$-torus, their graphs, and homotopy properties}\label{sec:function-graphs}
In this section we give a short overview of known results about graphs of smooth functions on $2$-torus, see \cite{MaksymenkoFeshchenko:2014:HomPropTree, MaksymenkoFeshchenko:2014:HomPropCycle, MaksymenkoFeshchenko:2015:HomPropCycleNonTri,Feshchenko:2014:HomPropTreeNonTri}.
Let $f$ be a Morse function on $T^2$, $\KR_f$ be its graph, and $p_f:T^2\to \KR_f$ be the projection map induced by $f$.

\subsection*{Case 1: Functions from $\mathscr{F}_0$}
The following lemma holds.
\begin{lemma}\label{lm:spec-vertex}{\rm\cite[Proposition 1]{MaksymenkoFeshchenko:2014:HomPropTree}.}
	Let $f$ be a function from $\mathscr{F}_0$, and $\Gamma_f$ be its graph.
	Then there exits a unique vertex $v$ in $\KR_f$ such that each component of the complement $T^2\setminus p_f^{-1}(v)$ is an open $2$-disk.
	Such a vertex $v$ of $\Gamma_f$ is called \myemph{special}.
	\qed
\end{lemma}
Let $v$ be a special vertex of $\KR_f$,
$\Gstab_v$ be the stabilizer of $v$ with respect to the action of the group $\Gstab(f)$ acting on $\KR_f$. Since $f$ has a unique special vertex, it follows that $G_v$ coincides with $G(f)$.  Let $\st(v)$ be a $\Gstab_v$-invariant connected neighborhood of $v$ containing no other vertices of $\KR_f$.
The set $\Gstab_v^{loc} = \{ g|_{\st(v)}\mid  g\in \Gstab_v \}$ consisting of restrictions of elements of $\Gstab_v$ onto $\st(v)$ is a subgroup of $\Aut(\st(v))$.
We will call $\Gstab_v^{loc}$ the {\it local stabilizer of $v$}.
Let also $r:\Gstab_v\to \Gstab_v^{loc}$ be the restriction map.
\begin{lemma}\label{lm:local-stab-acts}
	{\rm\cite[Theorem~2.5]{Feshchenko:2014:HomPropTreeNonTri}.}
	Let $f \in \mathscr{F}_0$, and $v$ be the special vertex of $\KR_f$.
	\begin{enumerate}[wide, label={\rm(\arabic*)}]
		\item\label{enum:lm:local-stab-acts:1}
		Then $\Gstab_v^{loc}$ is isomorphic to $\ZZZ_n\times\ZZZ_{\xxnm}$ for some $n,m\in \NNN$.
		\item\label{enum:lm:local-stab-acts:2}
		There exists a section $s:\Gstab_v^{loc}\to \Stab'(f)$ of the map $r\circ \rho$ such that $s(\Gstab_v^{loc})$ freely acts on $T^2$, so the map $q:T^2\to T^2/s(G_v^{loc})$ is a covering projection.
		Moreover the space of orbits $T^2/s(G_v^{loc})$ is also diffeomorphic to $T^2$.
	\end{enumerate}
\end{lemma}
\begin{remark}\label{rm:G-act}
	\begin{enumerate}[wide, label={\rm(\roman*)}]
		\item\label{rm:G-act:i}
		Throughout the paper by $G_v^{loc}$-action on $T^2$ we will mean its action via the section $s$ from~\ref{enum:lm:local-stab-acts:2}.
		So we will write $T^2/G_v^{loc}$ instead of $T^2/s(G_v^{loc})$.
		
		\item\label{rm:G-act:ii}
		Let $V = p_f^{-1}(v)$ be a critical component of a level set of $f$ which corresponds to the special vertex $v$.
		Since $G_v^{loc}$ freely acts on $T^2$, it follows that connected components of $T^2\setminus p_f^{-1}(v)$ can be enumerated by three indexes $D_{ijk}$, $i = 1,2,\ldots, r,$ $j = 0,\ldots, n-1$, and $k = 0,\ldots, {\xxnm}-1$.
		So, if $\gamma = (a,b)\in G_v^{loc} = \ZZZ_n\times\ZZZ_{{\xxnm}}$ then
		\[ \gamma(D_{ijk}) = D_{i,\ j + a\ \mod\  n,\ k+b\ \mod\  {\xxnm}}.\]

		\item\label{rm:G-act:iii}
		The set $\mathbb{D} = \{D_{i00}\}_{i = 1}^r$ from~\ref{enum:thm:main:1} of Theorem~\ref{thm:main} will be called {\it a fundamental set} of this $G_v^{loc}$-action.
		Numbers $n,m$ and $r$ in~\ref{enum:thm:main:1} of Theorem~\ref{thm:main} are also determined from~\ref{rm:G-act:i}.
	\end{enumerate}
\end{remark}

\subsection*{Case 2: Functions from $\mathscr{F}_1$}
Let $f\in\mathscr{F}_1$ and $\Theta$ be a unique circuit in $\Gamma_f$.
Let also $C_0\subset T^2$ be a regular connected component of some level set $f^{-1}(c)$, $c\in \RRR$, such that $z = p_f(C_0)$ is the point belonging to $\Theta$.

Notice that $f^{-1}(c)$ consists of finitely many connected components and is invariant under the action of any $h\in\Stab(f)$.
Let \[ \mathcal{C}=\{h(C_0) \mid  h\in\Stab'(f) \} \]
be the set of all images of $C_0$ under the action of elements from $\Stab'(f)$.
Evidently, the curves from $\mathcal{C}$ are pairwise disjoint.
Since $C_0$ does not separate $T^2$, it follows that each $C_i$ does not separate $T^2$ as well.
Two such curves will be called {\it parallel}, and the set $\mathcal{C}$ will be called {\it parallel family on $T^2$}.

We can also assume that they are cyclically enumerated along $T^2$, so that $C_i$ and $C_{i+1}$ bound a cylinder $Q_i$ such that the interior of $Q_i$ does not intersect $\mathcal{C}$, where all indexes are taken modulo $n$.
The number $n$ from~\ref{enum:thm:main:2} of Theorem~\ref{thm:main} is the number of curves in $\mathcal{C}$;  we will call it the {\it cyclic index of $f$}.
For other details see \cite{MaksymenkoFeshchenko:2015:HomPropCycleNonTri}.

\section{Preliminaries}
\label{sec:prelim}
\subsection{Homotopy properties of Morse functions on cylinder}
\label{subsec:homot}
The following lemma holds.
\begin{lemma}\label{lm:func-cylinder}
	{\rm\cite[Theorem 5.8, Lemma 5.1 III (a)]{Maksymenko:DefFuncI:2014}.}
	Let $f$ be a Morse function on cylinder $Q = S^1\times [0,1]$.
	Then there exists a family of mutually disjoint $2$-disks $\mathbb{Y} = \{Y_{ij}\}_{i = 1,\ldots, k}^{j = 1,\ldots, c_i}\subset Q$ for some $k,c_i\in \NNN$, and $m_i\in \NNN$, $i = 1,\ldots, k$ such that
	\begin{enumerate}[wide, label={\rm(\arabic*)}, itemsep=1ex]
		\item\label{enum:lm:func-cylinder:1}
		$f|_{Y_{ij}}$ is also a Morse function for each $j = 1,\ldots, c_i$, $i = 1,\ldots, k$;
		
		\item\label{enum:lm:func-cylinder:2}
		there exists an isomorphism $\zeta = (\zeta_1, \zeta_2,\zeta_3)$ of the following short exact sequences:
		\[
		\xymatrix{
			\pi_0\Delta'(f) \ar@{^{(}->}[r]^{j_0} \ar[d]_{\cong}^{\zeta_1} & 	\pi_0\mathcal{S}'(f) \ar@{->>}[r]^{\rho} \ar[d]_{\cong}^{\zeta_2}& G(f) \ar[d]_{\cong}^{\zeta_3}\\
			\Delta_{\mathbb{Y}}\ar@{^{(}->}[r]^{j_0}& \mathcal{S}_{\mathbb{Y}} \ar@{->>}[r]^{\rho} & G_{\mathbb{Y}},
		}
		\]
		where the first row is a sequence \eqref{eq:S'(f)} with $X = \varnothing$, and
		\begin{align*}
			\Delta_{\mathbb{Y}} &= \prod_{i = 1}^k \Bigl(\bigl(\prod_{j = 1}^{c_i} \Delta_{Y_{ij}}\bigr)^{m_i} \times m_i\ZZZ \Bigr), &
			\Delta_{Y_{ij}} &= \pi_0\Delta'(f|_{Y_{ij}}, \partial Y_{ij}),  \\
			\mathcal{S}_{\mathbb{Y}} &= \prod_{i = 1}^k \Bigl(\bigl(\prod_{j = 1}^{c_i} \mathcal{S}_{Y_{ij}} \bigr) \wr_{m_i}\ZZZ \Bigr), &
			\mathcal{S}_{Y_{ij}} &= \pi_0\mathcal{S}'(f|_{Y_{ij}}, \partial Y_{ij}), \\
			G_{\mathbb{Y}} &= \prod_{i = 1}^k \Bigl( \bigl( \prod_{j = 1}^{c_i} G_{Y_{ij}} \bigl) \wr\ZZZ_{m_i} \Bigr), &
			G_{Y_{ij}} &= G(f|_{Y_{ij}}, \partial Y_{ij}).
		\end{align*}
		
		\item\label{enum:lm:func-cylinder:3}
		The kernel of the homomorphism $\pi_0\mathcal{S}'(f,\partial Q)\to \pi_0\mathcal{S}'(f)$ induced by the inclusion is isomorphic to $\ZZZ$ and
		is generated by the isotopy class of
		$\tau_0\circ \tau_1^{-1}$, where $\tau_i\in S'(f)$ is a Dehn twist along $S^1\times \{i\}$, $i = 0,1$.
		Moreover, in the notation of~\ref{enum:lm:func-cylinder:2}
		$[\tau_0\circ \tau_1^{-1}]$ corresponds to the element
		\begin{equation}
			\label{eq:E_k}
			(E_1,m_1,\ldots, E_k,m_k)\in \prod_{i = 1}^k \Bigl( \bigl( \prod_{j = 1}^{c_i} \Delta_{Y_{ij}}\bigr)^{m_i} \times m_i\ZZZ \Bigr),
		\end{equation}
		where $E_i$ is the unit of $(\prod_{j = 1}^{c_i} \Delta_{Y_{ij}})^{m_i}$, $i = 1,\ldots, k$,
		\cite[Lemma 5.1 III (a)]{Maksymenko:DefFuncI:2014}.
	\end{enumerate}
\end{lemma}

\subsection{Dehn twists and slides} \label{sec:Delta'(f)}
Let $\alpha, \beta:[-1,1]\to [0,1]$ be two smooth functions such that $\alpha = 0$ on $[-1,-1/2]$ and $\alpha = 1$ on $[1/2,1]$, while $\beta = 0$ on $[-1, -2/3]\cup [2/3,1]$, and $\beta = 1$ on $[-1/3, 1/3]$.

Let $Q = S^1\times [-1,1]$ be a cylinder and $C = S^1\times 0$.
Define the following two diffeomorphisms of $Q$ by the formulas
\begin{align*}
	\tau(z,t) &= (ze^{\alpha(t)}, t), &
	\theta(z,t) &= (ze^{\beta(t)}, t),
\end{align*}
for $(z,t)\in S^1\times [-1,1]$.
The diffeomorphisms $\tau$ and $\theta$ are called a {\it Dehn twist} and a {\it slide} along the curve $C$ respectively. Note that $\tau$ is fixed on some neighborhood of $\partial Q$ and $\theta$ is fixed on some neighborhood of $C\cup\partial Q$.

Let $M$ be a smooth surface and $C \subset M$ be a simple closed curve.
Suppose that $C$ is a two-sided curve, i.e., $C$ has a regular neighborhood $W$ diffeomorphic to a cylinder $Q$.
Fix any diffeomorphism $\phi:Q\to W$ such that $\phi(S^1\times 0) = C$.
Since $\tau$ is fixed on some neighborhood of $\partial Q$, it is easy to see that $\phi\circ \tau\circ \phi^{-1}:W\to W$ extends by the identity to a unique diffeomorphism $\overline{\tau}$ of $M$.
Any diffeomorphism $h:M\to M$ isotopic to $\overline{\tau}$ or $\overline{\tau}^{-1}$ will be called a \myemph{Dehn twist along $C$}.

Similarly $\phi\circ \theta\circ \phi^{-1}:W\to W$ extends by the identity map to the diffeomorphism $\overline{\theta}$ of $M$.
Any diffeomorphism $h:M\to M$ fixed on some neighborhood of $C$ supported on some cylindrical neighborhood $W$ of $C$ and isotopic to $\overline{\theta}$ or $\overline{\theta}^{-1}$ relatively to some neighborhood of $C\cup \overline{M\setminus Q}$ will be called a \myemph{slide along $C$}.
For more details see \cite{MaksymenkoFeshchenko:2015:HomPropCycleNonTri}.

\subsection{Groups $\pi_0\Delta'(f)$ and their special subgroups}
Let $f$ be a Morse function from $\mathscr{F}_1$ with cyclic index $n$, so there exists an $\mathcal{S}'(f)$-invariant family \[\cC = \{ C_i\mid  i = 0,\ldots, n-1\}\] of connected components of the same level set of $f$ and neither of those curves separate $T^2$, $i = 0,\ldots, n-1$.
We let denote by $Q_i$ the cylinder bounded by curves $C_i$ and $C_{i+1 \ \mathrm{mod} \ n}$.

A neighborhood $V$ of $C\in\cC$ will be called \myemph{$f$-adapted} if
\begin{itemize}[leftmargin=5ex]
	\item $V$ is diffeomorphic to $S^1\times [0,1]$ via a diffeomorphism, say $\phi$,
	\item $\phi^{-1}(S^1\times \{t\})$ is a connected component of some level set of $f$, $t\in [0,1]$,
	\item $V$ does not contain critical points of $f$.
\end{itemize}

Fix an $f$-adapted neighborhood $\nV_i$ of $C_i$, $i = 0,\ldots, n-1$, so that
\begin{itemize}[label=--\ ]
	\item $\nV_i \cap \nV_j = \varnothing$ for $i\not=j$;
	\item for each $i,j$ there exists $h\in\mathcal{S}'(f)$ such that $h(\nV_i)=\nV_j$.
\end{itemize}
In particular, the union
\[ \sV = \bigcup_{i = 0}^{n-1}\nV_i\]
is $\mathcal{S}'(f)$-invariant.

Notice that by definition the group $\Delta'(f,\sV)$ consists of diffeomorphisms of $T^2$ which are isotopic to $\mathrm{id}_{T^2}$, fixed on an $f$-adapted neighborhood $\sV$ of $\mathcal{C}$ and induce trivial homeomorphisms of $\Gamma_f$.
Let $j:\Delta'(f,\sV)\to \Delta'(f)$ be the natural inclusion.
It is known and is easy to see that the homomorphism
\[ j_0:\pi_0\Delta'(f,\sV)\to \pi_0\Delta(f)\] induced by $j$ is an epimorphism, see~\cite[Lemma 5.1]{Maksymenko:DefFuncI:2014}.
Let also $\nW_i$  be  an $f$-adapted neighborhood of $C_i$ satisfying $\nV_i \subset \mathrm{Int} \nW_i$
for $i = 0,1\ldots, n-1$.
Put $\sW = \bigcup_{i = 0}^{n-1} \nW_i,$
Then, \cite[Corollary 7.2]{Maksymenko:DefFuncI:2014}, the natural inclusion
\[ \Delta'(f,\sV)\hookrightarrow \Delta'(f,\sW) 
\]
is the homotopy equivalence and groups $\pi_0\Delta'(f)$ and $\pi_0\Delta'(f, \sV)$ are abelian groups.


Recall that a  vector field $F$ on a smooth oriented surface $M$ is called Hamiltonian-like for a Morse function $f$ if the following conditions hold:
\begin{itemize}
	\item singular points of $F$ correspond to critical points of $f$,
	\item $f$ is constant along $F$,
	\item Let $z$ be a critical point of $f$. Then there exists a local coordinate system $(x,y)$ such that $f(z) = 0$, $f(x,y) = \pm x^2 \pm y^2$ near $z$, and in these coordinates $F$ has the form
	$F(x,y) = -f'_y \frac{\partial }{\partial x} + f'_x \frac{\partial}{\partial x}$.
\end{itemize}
By \cite[Lemma 5.1]{Maksymenko:AGAG:2006} for every Morse function $f:M\to \RRR$ there exists a Hamiltonian-like vector field of $f$ if $M$ is smooth and orientable surface.

Fix a Hamiltonian-like vector field $F$ for the given  function Morse function $f$ on $T^2$, and let $\bF$ be the flow of $F$.
Then $\sW$ is $\bF$-invariant and consists of periodic orbits.
Therefore one can assume that periods of all trajectories of $\bF$ is equal to $1$ on $\sW$.

Let $\theta_i:T^2\to T^2$ be a slide along $C_i$ supported in $\nW_i\setminus \nV_i$, $i = 0,\ldots,n-1$, and
\[ \theta = \theta_0\circ \theta_1\circ \ldots\circ\theta_{n-1}.\]
Then, \cite[Lemma 5.2]{MaksymenkoFeshchenko:2015:HomPropCycleNonTri}, there exists a smooth function $\sigma:T^2\to \RRR$ such that
\begin{itemize}
	\item  $\sigma$ is constant along trajectories of $\bF$,
	\item $\sigma = 1$ on $\sV$, $\sigma = 0$ on $T^2\setminus \sW$, and
	\item $\theta = \bF_{\sigma}$,
\end{itemize}
Then for $k\in \ZZZ$ we have $\theta^k = \bF_{k\sigma}$.
From this definition it immediately follows that $\theta$ belongs to $\Delta'(f, \sV)$.
A free abelian subgroup of $\pi_0\Delta'(f,\sV)$ generated by $\theta$ will be denoted by $\langle\theta \rangle$.
The following theorem is our main result of this section.

\begin{theorem}\label{thm:H-kerj0}
	For $f \in \mathscr{F}_1$ the following statements hold true.
	\begin{enumerate}[leftmargin=*, label={\rm\arabic*)}]
		\item\label{enum:thm:H-kerj0:1}
		$\ker j_0 \cong \langle \theta\rangle\cong \mathbb{Z}$. In other words, each $h\in\Delta'(f,\sV)$ is isotopic to $\theta^{k(h)}$ for some $k(h)\in\ZZZ$ relatively $\sV$ whenever $[h]\in \ker j_0.$
		
		\item\label{enum:thm:H-kerj0:2}
		The following exact sequence  of abelian groups
		\begin{equation}\label{equ:thm:H-kerj0::exact_seq}
			\xymatrix{
				\langle\theta\rangle \ \ar@{^{(}->}[r] & \	\pi_0\Delta'(f,\sV)\ar[r]^-{j_0} \ & \ \pi_0\Delta'(f)
			}
		\end{equation}
		splits.
		In particular, we there is an isomorphism \[ \pi_0\Delta'(f,\sV)\cong \langle\theta \rangle\times \pi_0\Delta'(f).\]
	\end{enumerate}
\end{theorem}
\begin{proof}
	\ref{enum:thm:H-kerj0:1}
	Let $h \in \Delta'(f,\sV)$.
	We have to show that if $[h]\in \ker j_0$, then $h$ is isotopic to $\theta^{k(h)}$ relatively $\sV$ for some $k(h)\in \ZZZ$.
	
	Recall that the identity component of $\Delta'(f)$ is $\mathcal{S}_{\id}(f)$, see~\cite[Lemma 4.1]{Maksymenko:DefFuncI:2014}.
	Hence if $h\in\Delta'(f,\sV)$ is such that $[h]\in\ker j_0$, then $h\in\mathcal{S}_{\id}(f)$, and by~\cite[Lemma 6.1]{Maksymenko:DefFuncI:2014} there exists a unique smooth function $\alpha:T^2\to \RRR$ such that $h = \bF_{\alpha}$.
	
	We claim that such $h$ admits some ``simplification'' on $T^2\setminus \sW$, i.e., $h$ can be deformed relatively $\sV$ to a diffeomorphism $h'$ such that $[h']\in\ker(j_0)$ and $h'$ is also fixed on $T^2\setminus \sW$.
	Indeed, fix any smooth function $\delta:T^2\to \RRR$ satisfying
	\begin{itemize}[itemsep=1ex]
		\item $\delta|_{\sV} = 0$,
		\item $\delta|_{T^2\setminus\sW} = 1$,
		\item $\delta$ is constant along trajectories of $\bF$,
	\end{itemize}
	and define the following homotopy $H^t:T^2\to T^2$, $t\in [0,1]$, by the formula
	\[ H^t(h) = \bF^{-1}_{t\delta\alpha}\circ h.\]
	By \cite[Lemma 6.1]{Maksymenko:DefFuncI:2014} a family $\{H^t\}$ is in fact an isotopy between $H^0 = h$, $h' := H^1(h) = \bF^{-1}_{\delta\alpha}\circ h$ and each diffeomorphism $H^t(h)$ belongs to $\Delta'(f,\sV)$.
	In particular, the isotopy classes of $h$ and $h'$ in $\pi_0\Delta'(f,\sV)$ coincide, and $h'$ is fixed on $T^2\setminus\sW$.
	
	So we can redenote $h'$ with $h$ and additionally assume further that $h$ is fixed on $T^2\setminus\sW$.
	The restriction of $\alpha$ and $h$ onto $\nV_i$ will be also denoted by $\alpha_i$ and $h_i$ respectively $i = 0,\ldots, n-1$.
	By assumptions $h$ is fixed on $\sV$, that is \[h_i(x) = \bF_{\alpha_i(x)}(x) = x,  \qquad x\in \nV_i.\]
	Since the period of trajectories of $\bF$ is equal to $1$ on $\sW$, it follows that $\alpha_i$ takes an integer value, say $k_i(h)\in\ZZZ$, depending on $h$ for each $x\in \nV_i$.
	
	We claim that in fact $\alpha_i$ takes the same value $k\in\ZZZ$ on each $\nV_i$, so all of those numbers $k_i(h)$ must coincide.
	Indeed, recall $h|_{Q_i}$ is isotopic relatively $\sV\cap Q_i$ to the Dehn twist $\tau^{a_i}$ supported on $\sV\cap Q_i$, where
	\[ a_i = \alpha(C_{i+1})-\alpha(C_i)= k_{i+1}(h)- k_i(h), \qquad i = 0,\ldots, n-1.\]
	However, as $h\in\Delta'(f,\sV)$, it follows that $h|_{Q_i}$ is isotopic relatively $\sV\cap Q_i$ to $\id_{Q_i}=\tau^{0}$ for each $i = 0,\ldots, n-1$, that is
	\[ k_{i+1}(h)- k_i(h) = a_i = 0, \]
	and so $k_{i+1}(h) = k_i(h)$ for each $i = 0,\ldots, n-1$.
	Thus we can denote the common values of all $k_i(h)$ simply by $k(h)$.
	
	Define now an isotopy $H^t(h):T^2\to T^2$ between $h = \bF_{\alpha}$ and $\theta^k = \bF_{k(h)\sigma}$ by the formula
	\[ H^t(h) = \bF_{(1-t)\alpha+tk(h)\sigma}, \qquad t\in [0,1].\]
	Then $H^t(h)$ is fixed on $\sV$ for all $t\in [0,1]$ and so each $h$ with $[h]\in\ker j_0$ is isotopic to $\theta^{k(h)}$ relatively to $\sV$.
	In other words, $\ker j_0 \subset \langle\theta\rangle$.
	
	The inverse inclusion is easy $\langle\theta\rangle\subset\ker j_0$, and so $\ker j_0=\langle\theta\rangle$ is an free abelian group generated by the element $\theta$.

	\ref{enum:thm:H-kerj0:2}
	Now it follows from~\ref{enum:thm:H-kerj0:1} and surjectivity of $j_0$ that we have a short exact sequence~\eqref{equ:thm:H-kerj0::exact_seq}:
	\[
	\xymatrix{
		\langle \theta\rangle \ \ar@{^{(}->}[r] & \ \pi_0\Delta'(f,\sV) \ \ar@{->>}[r]^{j_0}& \ \pi_0\Delta'(f).
	}
	\]
	Moreover, \cite{Maksymenko:AGAG:2006}, it consists of abelian groups and the group $\pi_0\Delta'(f)$ is a free abelian.
	Therefore the above sequence splits, i.e., there is an isomorphism $\pi_0\Delta'(f,\sV)\cong \pi_0\Delta'(f)\times \langle \theta\rangle$.
\end{proof}

\section{Orbits of Morse functions on 2-torus} \label{sec:Orbits}
The following theorem describes the fundamental groups of orbits of Morse functions on $2$-torus.
We showed that they can be computed using  zero homotopy groups  of stabilizers of restrictions of the given function onto subsurfaces of $T^2$ being $2$-disks and cylinders.
\begin{theorem}\label{thm-orbits-t2}
	{\rm\cite{MaksymenkoFeshchenko:2014:HomPropCycle, MaksymenkoFeshchenko:2014:HomPropTree, MaksymenkoFeshchenko:2015:HomPropCycleNonTri, Feshchenko:2014:HomPropTreeNonTri}.}
	Let $f$ be a Morse function on $T^2$, and $\Gamma_f$ be its graph.
	\begin{enumerate}[wide, label={\rm(\arabic*)}, itemsep=1ex]
		\item\label{enum:thm-orbits-t2:1}
		Assume that $f$ belongs to $\mathscr{F}_0$ and $G_v^{loc}\cong \ZZZ_n\times\ZZZ_{\xxnm}$ for some $n,m\in \NNN$.
		Then there exists a set of mutually disjoint $2$-disks $\mathbb{D} = \{D_i\}_{i = 1}^r\subset T^2$ for some $r\in \NNN$ which is a fundamental set of the free $G_v^{loc}$-action on $T^2$ such that the restriction $f|_{D_i}$, $i=1,\ldots,r$, is a Morse function, and there is an isomorphism
		\begin{equation}\label{equ:thm-orbits-t2:xi1}
			\xi_1: \mathcal{S}_{\mathbb{D}}\wr_{n,{\xxnm}}\ZZZ^2\to \pi_1\mathcal{O}_f(f),
		\end{equation}
		where $\mathcal{S}_{\mathbb{D}} := \prod\limits_{i = 1}^r \pi_0\mathcal{S}'(f|_{D_i}, \partial D_i)$. Moreover $r$ is the number of orbits of free $G_v^{loc}$-action on $T^2$.

		\item\label{enum:thm-orbits-t2:2}
		Assume that $f$ belongs to $\mathscr{F}_1$ and has a cyclic index $n\in\NNN$.
		Then there exists a cylinder $Q\subset T^2$ such that $f|_{Q}$ is also a Morse function and we have an isomorphism
		\begin{equation}\label{equ:thm-orbits-t2:xi2}
			\xi_2: \mathcal{S}_Q\wr_{n}\ZZZ\to \pi_1\mathcal{O}_f(f),
		\end{equation}
		where $\mathcal{S}_Q:= \pi_0\mathcal{S}'(f|_{Q}, \partial Q)$.
	\end{enumerate}
\end{theorem}
\subsection{Strategy of the proof of Theorem \ref{thm:main}}\label{subsec:Strat}
First we recall explicit definitions of isomorphisms
\begin{align*}
	&\xi_1:\mathcal{S}_{\mathbb{D}}\wr_{n,\xxnm}\ZZZ^2\to \pi_1\mathcal{O}_f(f), &
	&\xi_2:\mathcal{S}_Q\wr_n\ZZZ \to \pi_1\mathcal{O}_f(f)
\end{align*}
from Theorem~\ref{thm-orbits-t2} in Subsections~\ref{subsec:xi_1} and~\ref{subsec:xi_2} respectively.
Since groups $\pi_0\mathcal{S}'(f)$ and $G(f)$ are quotient-groups
\begin{gather*}
	\pi_0\mathcal{S}'(f)\cong \pi_1\mathcal{O}_f(f)/\pi_1\mathcal{D}_{\id}(T^2), \\
	G(f)\cong \pi_1\mathcal{O}_f(f)/(\pi_1\mathcal{D}_{\id}(T^2)\times \pi_0\Delta'(f)),
\end{gather*}
see diagram \eqref{diag:main-diag}, it follows that in order to describe them and the diagram \eqref{diag:main-diag} we need to characterize images of $\pi_1\mathcal{D}_{\id}(T^2)$ and $\pi_0\Delta'(f)$ with respect to the maps $\xi_1^{-1}$ and $\xi_2^{-1}$ and take the corresponding quotient-groups.
This will be done in Subsection \ref{subsec:images} and Subsection \ref{subsec:images2}.

\section{Proof of~\ref{enum:thm:main:1} of Theorem~\ref{thm:main}}
\label{sec:proof-1}

Due to our strategy, see Subsection~\ref{subsec:Strat}, first we give the explicit description of the isomorphism $\xi_1:\mathcal{S}_{\mathbb{D}}\wr_{n,\xxnm}\ZZZ^2\to \pi_1\mathcal{O}_f(f)$.
\subsection{Isomorphism from~\ref{enum:thm-orbits-t2:1} of Theorem \ref{thm-orbits-t2}}
\label{subsec:xi_1}

Let $f$ be a Morse function on $T^2=\RRR^2/\ZZZ^2$ such that its graph $\Gamma_f$ is a tree, $v$ be the special vertex of $\Gamma_f$, and $G_v^{loc} = \ZZZ_n\times \ZZZ_{\xxnm}$ be the local stabilizer of $v$.
Then the free action of $\ZZZ_n\times \ZZZ_{\xxnm}$ on $T^2$ can be given by:
\[
\kappa_{(a,b)}(x,y) = \Bigl(x + \frac{a}{n} \ \mod\ 1, \ y + \frac{b}{\xxnm} \ \mod\ 1 \Bigr),
\]
for $(a,b)\in\ZZZ_n\times \ZZZ_{nm}$ and $(x,y)\in T^2$, and the quotient space $T^2/G_v^{loc}$ is diffeomorphic to $T^2 = \RRR^2/\ZZZ^2$, so that the quotient map $q:T^2\to T^2/G_v^{loc}$ is given by the formula
\[
q(x,y) = (nx \ \mod\ 1, \ \xxnm y \ \mod\ 1).
\]

Let $y$ be a point in $T^2$, and $z = q(y)\in T^2/G_v^{loc}$.
Then we obtain the following commutative diagram with exact rows
\begin{equation}\label{eq:cover}
	\begin{aligned}
		\xymatrix{
			\pi_1 (T^2,y) \ar@{^{(}->}[rr]^-{q_1} \ar@{=}[d] && \pi_1 (T^2/G_v^{loc}, z) \ar@{->>}[rr]^-{\partial_1} \ar@{=}[d]&& G_v^{loc} \ar@{=}[d]\\
			\ZZZ^2 \ar@{^{(}->}[rr]^-{q_1}                   && \ZZZ^2\ar@{->>}[rr]^-{\partial}                                && \ZZZ_n\times \ZZZ_{\xxnm},
		}
	\end{aligned}
\end{equation}
where maps $q_1$ and $\partial$ are given as follows
\begin{align*}
	q_1(\lambda, \mu) &= (n\lambda, \ {\xxnm}\mu), &
	\partial(x,y) &= (x\ \mod\ n, \ y\ \mod\ {\xxnm}).
\end{align*}

Let $\mathsf{L},\mathsf{M}:T^2\times\RRR\to T^2$ be smooth flows on $T^2$ defined by
\begin{align*}
	\mathsf{L}(x,y,t) &=\Bigl( x+\frac{t}{n} \ \mod \ 1, \ y, \ t\Bigr), &
	\mathsf{M}(x,y,t) &=\Bigl( x, \ y + \frac{t}{\xxnm} \ \mod \ 1, \ t\Bigr).
\end{align*}
Then they commute each with other, and
\begin{align*}
	\mathsf{L}_a(x,y) &= \kappa_{(a,0)} (x,y), &
	\mathsf{M}_b(x,y) &= \kappa_{(0,b)} (x,y),
\end{align*}
for $a,b\in \RRR$.
Moreover, $\mathsf{L}_n=\mathsf{M}_{\xxnm} = \id_{T^2}$, and the restrictions
\begin{align*}
	&\mathsf{L}: T^2\times[0,n] \to T^2, &
	&\mathsf{M}: T^2\times[0,\xxnm] \to T^2
\end{align*}
can be regarded as loops in $\mathcal{D}_{\id}(T^2)$ constituting also a basis $(1,0)$, $(0,1)$ of  $\pi_1\mathcal{D}_{\id}(T^2) \cong \ZZZ^2$.

Then, \cite{Feshchenko:2014:HomPropTreeNonTri}, the isomorphism $\xi_1: \mathcal{S}_{\mathbb{D}}\wr_{n,{\xxnm}}\ZZZ^2\to \pi_1\mathcal{O}_f(f)$, see~\eqref{equ:thm-orbits-t2:xi1}, can be defined as follows.
Let
\[ \bigl( \{ h_{ijk}\}, (a,b) \bigr) \in \mathcal{S}_{\mathbb{D}}\wr_{n,{\xxnm}}\ZZZ^2,\]
where $h_{ijk}\in \mathcal{S}'(f|_{D_{i00}}, \partial D_{i00})$, $(a,b)\in \ZZZ^2$, $i = 1,\ldots, r$, $j = 0,\ldots, n-1$, $k = 0,\ldots,\xxnm-1$.
For each triple $(i,j,k)$ fix any isotopy $h^t_{ijk}:D_{i00}\to D_{i00}$, $t\in [0,1]$ between $h^0_{ijk} = \id_{D_{i00}}$ and $h^1_{ijk} = h_{ijk}$ relatively some neighborhood of $\partial D_{i00}$.
Then
\begin{equation}
	\label{eq:xi_1}
	\xi_1\bigl( \{ h_{ijk}\}, (a,b) \bigr) = [\{ f\circ h^t \}],
\end{equation}
where $h^t:T^2\to T^2$, $t\in[0,1]$, is given by the formula
\begin{equation}
	\label{eq:h(x)}
	h^t(x) = \begin{cases}
		\{\mathsf{M}_{k+\frac{bt}{nm}}\circ \mathsf{L}_{j+ \frac{at}{n}}\circ h^t_{ijk}\circ \mathsf{L}_j^{-1}\circ \mathsf{M}_k^{-1}(x)\}_{ijk},& x\in D_{ijk}, \\[2ex]
		\mathsf{M}_{ \frac{bt}{nm}}\circ \mathsf{L}_{\frac{at}{n}}(x),& x\in N,
	\end{cases}
\end{equation}
$D_{ijk}$ is a $2$-disk defined in~\ref{rm:G-act:ii} of Remark~\ref{rm:G-act}, and $N$ is a regular neighborhood of the critical level-set $V=p_f^{-1}(v)$ containing no other critical points.

\subsection{Images of $\pi_1\mathcal{D}_{\id}(T^2)$ and $\pi_0\Delta'(f)$ in $\mathcal{S}_{\mathbb{D}}\wr_{n,\xxnm}\ZZZ^2$}\label{subsec:images}
The following lemma describes the images of the groups $\pi_1\mathcal{D}_{\id}(T^2)$ and $\pi_0\Delta'(f)$ with respect to the map $\xi_1^{-1}$.
\begin{lemma}\label{lm:ZZ}
	\begin{enumerate}[leftmargin=*, label={\rm(\arabic*)}]
		\item\label{enum:lm:ZZ:1}
		Let $Z_{n,0} = (\underbrace{e,\ldots, e}_{{\xxnm}n}, n,0)$ and $Z_{0,nm} = (\underbrace{e,\ldots, e}_{{\xxnm}n}, 0,nm)$ be elements from $\mathcal{S}_{\mathbb{D}}\wr_{n,\xxnm}\ZZZ^2$, where $e$ is the unit of $\mathcal{S}_{\mathbb{D}}$.
		Then
		\[
		p_1 (\mathsf{L}) = \xi_1(Z_{n,0}), \qquad  p_1 (\mathsf{M}) = \xi_1(Z_{0,nm}).
		\]
		
		\item\label{enum:lm:ZZ:2}
		The isomorphism $\xi_1$ induces an isomorphism 
		\[\xi_1|_{\Delta^{\xxnmn}_{\mathbb{D}}}:(\Delta^{\xxnmn}_{\mathbb{D}}, 0,0)\to \iota_1(\pi_0\Delta'(f)).\]
	\end{enumerate}
\end{lemma}
\begin{proof}
	\ref{enum:lm:ZZ:1}
	Indeed, $\xi_1(Z_{n,0})$ is a loop $f\circ h^t,$ where $h^t$ is given by~\eqref{eq:h(x)} with $a = 1$ and $b = 0$.
	Then by~\eqref{eq:h(x)} the isotopy $h$ has the form $h^t = \mathsf{L}_{t}$.
	
	The case of $p_1(\mathsf{M})$ can be checked similarly and  we leave this verification to the reader.
	So we have $p_1 (\langle\mathsf{L}\rangle) = \xi_1(\langle Z_{n,0}\rangle)$ and $p_1 (\langle\mathsf{M}\rangle) = \xi_1(\langle Z_{0,nm}\rangle)$.

	\ref{enum:lm:ZZ:2}
	Let $\iota: \pi_0\Delta'(f)\to \pi_1\mathcal{O}_f(f)$ be a restriction of $\iota_1$ onto ${\pi_0\Delta'(f)}$.
	We need to show that the subgroup $(\Delta^{\xxnmn}_{\mathbb{D}}, 0,0)\subset \mathcal{S}_{\mathbb{D}}\wr_{n,\xxnm}\ZZZ^2$ is mapped via $\xi_1$ isomorphically to $\iota(\pi_0\Delta'(f))$.
	Obviously, an isomorphism $\xi_1$ induces a monomorphism $\xi_1|_{\Delta^{\xxnmn}_{\mathbb{D}}}: \Delta^{\xxnmn}_{\mathbb{D}} \to \iota_1(\pi_0\Delta'(f))$.
	It remains to show that $\xi_1|_{\Delta^{\xxnmn}_{\mathbb{D}}}$ is an epimorphism.
	
	Let $[\tau]$ be a generator of $\pi_0\Delta'(f)$.
	Then $\tau$ be a Dehn twist supported on some $2$-disk $D\in \{D_{i,j,k}\}_{i = 1,\ldots, r,\,j = 0,\ldots,n-1,\, k = 0,\ldots, \xxnm-1}$.
	Fix an isotopy $\tau^t:T^2\to T^2$ between $\tau^1 = \tau$ and $\tau^0 = \id_{T^2}$.
	The image $\iota([\tau])$ is given by the class of the loop $[f\circ \tau^t]$.
	It follows from~\eqref{eq:h(x)} that there exist a diffeomorphism $\tilde{h}$ of $D$ and an isotopy $\tilde{h}^{t}:D\to D$ between $\tilde{h}^1 = \tilde{h}$ and $\tilde{h}^0 = \id_D$ such that $\xi_1(\{[\tilde{h}]\}, 0,0) = [f\circ \tau^t]$.
	Also we get from~\eqref{eq:h(x)} that $\tilde{h}$ is conjugate to $\tau|_D$, and  so $\tilde{h}$ is a Dehn twist.
	Hence $\tilde{h}$ belongs to $\Delta_{\mathbb{D}}^{\xxnmn}$.
	Thus $\xi_1|_{\Delta^{\xxnmn}_{\mathbb{D}}}$ is an isomorphism.
\end{proof}

\subsection{Final remarks for the proof of~\ref{enum:thm:main:1} of Theorem~\ref{thm:main}}
Recall that each third term of every short exact sequence of groups is uniquely defined by other two.
So groups $\pi_0\mathcal{S}'(f)$ and $G(f)$ are the corresponding quotient-groups $\pi_1\mathcal{O}_f(f)/\pi_1\mathcal{D}_{\id}(T^2)$ and $G(f) = \pi_0\mathcal{S}'(f)/\pi_0\Delta'(f)$, see diagram~\eqref{diag:main-diag}.
The images of $\pi_1\mathcal{D}_{\id}(T^2)$ and $\pi_0\Delta'(f)$ in $\pi_1\mathcal{O}_f(f)$ are known, see Subsection~\ref{subsec:images}.
So it is easy to prove that an isomorphism $\xi_1$ induces the following isomorphisms
\begin{align*}
	&\pi_0\mathcal{S}'(f)\cong \mathcal{S}_{\mathbb{D}}\wr (\ZZZ_n\times \ZZZ_{nm}), &
	&G(f)\cong G_{\mathbb{D}}\wr(\ZZZ_n\times\ZZZ_{\xxnm}),
\end{align*}
as well as isomorphism of diagrams from~\ref{enum:thm:main:1} of Theorem~\ref{thm:main}.

\section{Proof of~\ref{enum:thm:main:2} of Theorem~\ref{thm:main}}
\label{sec:proof-2}
First we give the explicit description of an isomorphism  $\xi_2:\mathcal{S}_Q\wr_n\ZZZ \to \pi_1\mathcal{O}_f(f)$.
\subsection{Isomorphism from~\ref{enum:thm-orbits-t2:2} of Theorem~\ref{thm-orbits-t2}}
\label{subsec:xi_2}
Let $f$ be a Morse function from $\mathscr{F}_1$ with a cyclic index $n$ and circuit $\Theta$ in the graph $\Gamma_f$.
As curves from $\mathcal{C}$ are ``parallel'', one can assume that the following conditions hold:

(a) $C_i = \frac{i}{n}\times S^1\subset \RRR^2/\ZZZ^2 = T^2$;

(b) there exists $\varepsilon\geq 0$ such that for all $t\in (\frac{i}{n}-\varepsilon, \frac{i}{n}+\varepsilon)$ the curve $t\times S^1$ is a regular connected component of some level set of $f$.

Let $\mathbb{L}, \mathbb{M}: T^2\times \RRR\to T^2$ be two flows defined by formulas:
\begin{align*}
	\mathbb{L}_t(x,y) &= (x+t\ \mod\  1,\ y), &
	\mathbb{M}_t(x,y) &= (x,\ y + t\ \mod\ 1),
\end{align*}
$x\in C'$, $y\in C_0$, and $t\in \RRR$.

Denote by $Q$  the cylinder bounded by $C_0$ and $C_1$.
Then the isomorphism $\xi_2:\mathcal{S}_{Q}\wr_n\ZZZ\to \pi_1\mathcal{O}_f(f)$ can be defined as follows, \cite[Section 8]{MaksymenkoFeshchenko:2015:HomPropCycleNonTri}.
Let
\[
(h_1,\ldots,h_n; a) \in \mathcal{S}_{Q}\wr_n\ZZZ,
\]
where $h_i\in \mathcal{S}_Q$, $i = 0,\ldots, n$, and $a\in\ZZZ$.
Fix any isotopy $h_i^t$ between $h^0_i = \id_{Q}$ and $h^1_i = h_i$.
Then
\begin{equation}
	\label{eq:xi_2}
	\xi_2(h_1^t,\ldots,h_n^t; a)  = [f\circ h^t],
\end{equation}
where $h^t$ is defined by the formula
\begin{equation}
	\label{eq:hh(x)}
	h^t(x) = \{\mathbb{L}_{i + \frac{at}{n}}\circ h^t_i\circ \mathbb{L}_i^{-1}(x)\}_{i = 0,\ldots, n},\quad x\in Q_i
\end{equation}

\subsection{Images of $\pi_1\mathcal{D}_{\id}(T^2)$ and $\pi_0\Delta'(f)$ in $\mathcal{S}_Q\wr_n\ZZZ$}
\label{subsec:images2}
The following lemma easily follows from the definition of the isomorphism $\xi_2$ similarly to~\ref{enum:lm:ZZ:1} of Lemma \ref{lm:ZZ}.
\begin{lemma}
	Let $Z =(\underbrace{e',\ldots,e'}_n,n)$ be the element from $S_Q\wr_n\ZZZ$, where $e'$ is the unit of $\mathcal{S}_Q$.
	Then $p_1(\mathbb{L}) = \xi_2(Z)$.
	\qed
\end{lemma}
The image of another generator $\mathbb{M}$ in $\pi_1\mathcal{O}_f(f)$ is ``invisible'' in our description of the group $\pi_1\mathcal{O}_f(f)$ via $\mathcal{S}_Q\wr_n\ZZZ$, but in \cite[Theorem 6]{MaksymenkoFeshchenko:2015:HomPropCycleNonTri} we showed that $p_1(\mathbb{M})$ in $\pi_1\mathcal{O}_f(f)$ is given by $[f\circ \theta]$, where $\theta = \theta_0\circ \ldots\circ\theta_{n-1}$ is the generator of $H = \langle\theta\rangle$, and $\theta_i$ is a slide along $C_i$, see Section~\ref{sec:Delta'(f)}.
Then by Theorem~\ref{thm:H-kerj0} groups $\pi_0\Delta'(f)\times \langle\mathbb{M} \rangle$ and $\pi_0\Delta'(f,\mathcal{C})$ are isomorphic as subgroups of $\pi_1\mathcal{O}_f(f)$.
So $\pi_0\Delta'(f)\cong \pi_0\Delta'(f,\mathcal{C})/H$.
\begin{lemma}
	An isomorphism $\xi_2$ induces an isomorphism 
	\[ \xi_2|_{\Delta_Q^n}:(\Delta_Q^n,0)\to \iota_1({\pi_0\Delta'(f,\mathcal{C})}),\]
	where $\Delta_Q = \pi_0\Delta'(f|_Q,\partial Q)$.
	So $\pi_0\Delta'(f)$ is isomorphic to $\Delta_Q^n/H$.
\end{lemma}
\begin{proof}
	Obviously, the isomorphism $\xi_2$ induces a monomorphism $\xi_2|_{\Delta_Q^n}$.
	Let $[\tau]$ be a generator of $\pi_0\Delta'(f,\mathcal{C})$.
	Then $\tau$ is a Dehn twist supported in $Q\in \{Q_i\}_{i = 0,\ldots, n-1}$.
	Fix an isotopy $\tau^t:T^2\to T^2$ such that $\tau^1 = \tau$ and $\tau^0 = \id$.
	Then $\iota([\tau]) = [f\circ \tau^t]$.
	From~\eqref{eq:xi_2} there exists a diffeomorphism $\tilde{h}:Q\to Q$ and an isotopy $\tilde{h}^t:Q\to Q$ between $\tilde{h}^1 = \tilde{h}$ and $\tilde{h}^0 = \id_Q$ such that $\xi_2([\tilde{h}], 0) = [f\circ \tau^t]$.
	By~\eqref{eq:hh(x)} $\tilde{h}$ is conjugate to $\tau_Q$, then $\tilde{h}$ is also a Dehn twist.
	Hence $\tilde{h}$ belongs to $\Delta_Q^n$.
	So $\xi_2|_{\Delta_Q^n}$ is an isomorphism, and hence $\pi_0\Delta'(f)$ is isomorphic to $\Delta_Q^n/H$.
\end{proof}

The following corollary directly follows from Lemma~\ref{lm:func-cylinder} and describes $\pi_1\mathcal{O}_f(f)$ for functions from $\mathscr{F}_1$ via embedded $2$-disks and the image of the generator $\mathbb{M}$ in $\pi_1\mathcal{O}_f(f)$.

\begin{corollary}\label{cor:pi1-disks}
	\begin{enumerate}[wide, label={\rm(\arabic*)}]
		\item\label{cor:pi1-disks:1}
		There exist a cylinder $Q$ and a set of mutually disjoint $2$-disks $\mathbb{Y} = \{Y_{ij}\}_{i = 0,\ldots, k}^{j = 0,\ldots, c_i}\subset Q$ for some $k,c_i\in \NNN$, and $m_i\in \NNN$, $i = 1,\ldots, k$ such that the following groups are isomorphic
		\[
		\pi_1\mathcal{O}_f(f) \stackrel{\xi_2^{-1}}{\cong}\mathcal{S}_Q\wr_n\ZZZ \stackrel{\zeta'}{\cong} \mathcal{S}_{\mathbb{Y}}\wr_n\ZZZ,
		\]
		where $\zeta' = \underbrace{\zeta\times\ldots\times\zeta\times\id_{\ZZZ}}.$
		
		\item\label{cor:pi1-disks:2}
		The group $H$ is normal in $\Delta_{\mathbb{Y}}^n$ is generated by
		\begin{equation}
			\label{eq:H}
			(\underbrace{(E_1, m_1,E_2,m_2,\ldots, E_k, m_k), \ldots, (E_1, m_1,E_2,m_2,\ldots, E_k, m_k)}_{n},
		\end{equation}
		where
		$E_i$ is the unit of the group $(\prod_{j = 1}^{c_i} \Delta_{Y_{ij}})^{m_i}$, and the image of $H$ in $ \mathcal{S}_{\mathbb{Y}}\wr_n\ZZZ$ is generated by
		\begin{equation}
			\label{eq:overE}
			(\underbrace{(E_1, m_1,E_2,m_2,\ldots, E_k, m_k), \ldots, (E_1, m_1,E_2,m_2,\ldots, E_k, m_k)}_{n}, 0),
		\end{equation}
	\end{enumerate}
\end{corollary}

\subsection{Final remarks for the  proof of~\ref{enum:thm:main:2} of Theorem~\ref{thm:main}}
Similarly to the proof of~\ref{enum:thm:main:1} of Theorem~\ref{thm:main}, groups $\pi_0\mathcal{S}'(f)$ and $G(f)$ are the corresponding  quotient-groups $\pi_1\mathcal{O}_f(f)/\pi_1\mathcal{D}_{\id}(T^2)$ and $G(f) = \pi_0\mathcal{S}'(f)/\pi_0\Delta'(f)$, see diagram \eqref{diag:main-diag}.
The images of $\pi_1\mathcal{D}_{\id}(T^2)$ and $\pi_0\Delta'(f)$ in $\pi_1\mathcal{O}_f(f)$ are known, see Subsection~\ref{subsec:images2}.
So it is easy to prove that an isomorphism $\zeta'\circ \xi_2$ induces the following isomorphisms
\begin{align*}
	&\pi_0\mathcal{S}'(f)\cong (\mathcal{S}_{\mathbb{Y}}\wr \ZZZ_n)/H, &
	&G(f)\cong G_{\mathbb{Y}}\wr\ZZZ_n,
\end{align*}
and so it induces an isomorphism of diagrams from~\ref{enum:thm:main:2} of Theorem~\ref{thm:main}.

%

\begin{thebibliography}{10}
	
	\bibitem{EarleEells:DG:1970}
	C.~J. Earle and J.~Eells.
	\newblock A fibre bundle description of teichm\"uller theory.
	\newblock {\em J. Differential Geometry}, 3:19--43, 1969.
	
	\bibitem{Feshchenko:2014:HomPropTreeNonTri}
	B.~Feshhenko.
	\newblock Deformations of smooth functions on $2$-torus, whose kr-graph is a
	tree.
	\newblock {\em Proceedings of Institute of Mathematics of NAS of Ukraine},
	12(6):22--40, 2015.
	
	\bibitem{Gramain:ASENS:1973}
	Andr{\'e} Gramain.
	\newblock Le type d'homotopie du groupe des diff\'eomorphismes d'une surface
	compacte.
	\newblock {\em Ann. Sci. \'Ecole Norm. Sup. (4)}, 6:53--66, 1973.
	
	\bibitem{IkegamiSaeki:JMSJap:2003}
	Kazuichi Ikegami and Osamu Saeki.
	\newblock Cobordism group of {M}orse functions on surfaces.
	\newblock {\em J. Math. Soc. Japan}, 55(4):1081--1094, 2003.
	
	\bibitem{Kalmar:KJM:2005}
	Boldizs{\'a}r Kalm{\'a}r.
	\newblock Cobordism group of {M}orse functions on unoriented surfaces.
	\newblock {\em Kyushu J. Math.}, 59(2):351--363, 2005.
	
	\bibitem{Kudryavtseva:MatSb:1999}
	E.~A. Kudryavtseva.
	\newblock Realization of smooth functions on surfaces as height functions.
	\newblock {\em Mat. Sb.}, 190(3):29--88, 1999.
	
	\bibitem{Kudryavtseva:MathNotes:2012}
	E.~A. Kudryavtseva.
	\newblock The topology of spaces of {M}orse functions on surfaces.
	\newblock {\em Math. Notes}, 92(1-2):219--236, 2012.
	\newblock Translation of Mat. Zametki {{\bf{9}}2} (2012), no. 2, 241--261.
	
	\bibitem{Kudryavtseva:MatSb:2013}
	E.~A. Kudryavtseva.
	\newblock On the homotopy type of spaces of {M}orse functions on surfaces.
	\newblock {\em Mat. Sb.}, 204(1):79--118, 2013.
	
	\bibitem{MaksymenkoFeshchenko:2014:HomPropTree}
	S.~Maksymenko and B.~Feshchenko.
	\newblock Homotopy properties of spaces of smooth functions on$2$-torus.
	\newblock {\em Ukrainian Mathematical Journal}, 66(9):1205--1212, 2014.
	
	\bibitem{Maksymenko:AGAG:2006}
	Sergiy Maksymenko.
	\newblock Homotopy types of stabilizers and orbits of {M}orse functions on
	surfaces.
	\newblock {\em Ann. Global Anal. Geom.}, 29(3):241--285, 2006.
	
	\bibitem{Maksymenko:MFAT:2010}
	Sergiy Maksymenko.
	\newblock Functions on surfaces and incompressible subsurfaces.
	\newblock {\em Methods Funct. Anal. Topology}, 16(2):167--182, 2010.
	
	\bibitem{Maksymenko:ProcIM:ENG:2010}
	Sergiy Maksymenko.
	\newblock Functions with isolated singularities on surfaces.
	\newblock {\em Geometry and topology of functions on manifolds. Pr. Inst. Mat.
		Nats. Akad. Nauk Ukr. Mat. Zastos.}, 7(4):7--66, 2010.
	
	\bibitem{Maksymenko:OsakaJM:2011}
	Sergiy Maksymenko.
	\newblock Local inverses of shift maps along orbits of flows.
	\newblock {\em Osaka Journal of Mathematics}, 48(2):415--455, 2011.
	
	\bibitem{Maksymenko:UMZ:ENG:2012}
	Sergiy Maksymenko.
	\newblock Homotopy types of right stabilizers and orbits of smooth functions
	functions on surfaces.
	\newblock {\em Ukrainian Math. Journal}, 64(9):1186--1203, 2012.
	
	\bibitem{Maksymenko:DefFuncI:2014}
	Sergiy Maksymenko.
	\newblock Deformations of functions on surfaces by isotopic to the identity
	diffeomorphisms.
	\newblock page arXiv:math/1311.3347, 2014.
	
	\bibitem{MaksymenkoFeshchenko:2015:HomPropCycleNonTri}
	Sergiy Maksymenko and Bohdan Feshchenko.
	\newblock Functions on $2$-torus whose kronrod-reeb graph contains a cycle.
	\newblock {\em Methods of Functional Analysis and Topology}, 21(1):22--40,
	2015.
	
	\bibitem{MaksymenkoFeshchenko:2014:HomPropCycle}
	Sergiy Maksymenko and Bohdan Feshchenko.
	\newblock Orbits of smooth functions on $2$--torus and their homotopy types.
	\newblock {\em Matematychni Studii}, 44(1):67--83, 2015.
	
	\bibitem{Meldrum:Longman:1995}
	J.~D.~P. Meldrum.
	\newblock {\em Wreath products of groups and semigroups}, volume~74 of {\em
		Pitman Monographs and Surveys in Pure and Applied Mathematics}.
	\newblock Longman, Harlow, 1995.
	
	\bibitem{Sergeraert:ASENS:1972}
	Francis Sergeraert.
	\newblock Un th\'eor\`eme de fonctions implicites sur certains espaces de
	{F}r\'echet et quelques applications.
	\newblock {\em Ann. Sci. \'Ecole Norm. Sup. (4)}, 5:599--660, 1972.
	
	\bibitem{Sharko:PrIntMat:1998}
	V.~V. Sharko.
	\newblock Functions on surfaces. {I}.
	\newblock In {\em Some problems in contemporary mathematics ({R}ussian)},
	volume~25 of {\em Pr. Inst. Mat. Nats. Akad. Nauk Ukr. Mat. Zastos.}, pages
	408--434. Nats\={i}onal. Akad. Nauk Ukra\"{i}ni, \={I}nst. Mat., Kiev, 1998.
	
\end{thebibliography}

\end{document}